\documentclass[12pt]{amsart}
\usepackage{a4wide}
\usepackage{graphicx,eepic, bm, times}
\usepackage{amsthm,amsmath,amsfonts,amssymb}
\usepackage{multirow}

\newtheorem{lemma}{Lemma}[section]

\newtheorem{prop}[lemma]{Proposition}
\newtheorem{thm}[lemma]{Theorem}

\theoremstyle{definition}

\theoremstyle{remark}

\numberwithin{equation}{section} \numberwithin{table}{section}

\begin{document}

\title[On the distribution of powers of real numbers modulo 1]{On the distribution of powers of real numbers modulo 1}
\author{Simon Baker}
\address{
School of Mathematics, The University of Manchester,
Oxford Road, Manchester M13 9PL, United Kingdom. E-mail:
simonbaker412@gmail.com}

\date{\today}

\begin{abstract}
Given a strictly increasing sequence of positive real numbers tending to infinity $(q_{n})_{n=1}^{\infty}$, and an arbitrary sequence of real numbers $(r_{n})_{n=1}^{\infty}.$ We study the set of $\alpha\in(1,\infty)$ for which $\lim_{n\to\infty}\|\alpha^{q_{n}}-r_{n}\|= 0$. In \cite{Dub} Dubickas showed that whenever $\lim_{n\to\infty}(q_{n+1}-q_{n})=\infty,$ there always exists a transcendental $\alpha$ for which $\lim_{n\to\infty}\|\alpha^{q_{n}}-r_{n}\|= 0.$ Adapting the approach of Bugeaud and Moshchevitin \cite{BugMos}, we improve upon this result and show that whenever $\lim_{n\to\infty}(q_{n+1}-q_{n})=\infty,$ the set of $\alpha\in(1,\infty)$ satisfying $\lim_{n\to\infty}\|\alpha^{q_{n}}-r_{n}\|= 0$ is a dense set of Hausdorff dimension $1$.
\end{abstract}

\maketitle

\section{Introduction}
It is a well known result of Koksma that for almost every $\alpha\in(1,\infty)$ the sequence $(\{\alpha^{n}\})_{n=1}^{\infty}$ is uniformly distributed modulo $1$ \cite{Kok}. Here and throughout almost every is meant with respect to the Lebesgue measure, and $\{\cdot\}$ denotes the fractional part of a real number. It is a long standing problem to determine the set of $\alpha\in(1,\infty)$ for which $\lim_{n\to\infty}\|\alpha^{n}\|= 0,$ where $\|\cdot\|$ denotes the distance to the nearest integer. The only known examples of numbers satisfying this property are Pisot numbers, that is algebraic integers whose Galois conjugates all have modulus strictly less than $1$. It was shown independently by Hardy \cite{Hardy} and Pisot \cite{Pisot2}, that if $\alpha$ is an algebraic number and $\lim_{n\to\infty}\|\alpha^{n}\|= 0$, then $\alpha$ is a Pisot number. Moreover, Pisot in \cite{Pisot} showed that there are only countably many $\alpha\in(1,\infty)$ satisfying $\lim_{n\to\infty}\|\alpha^{n}\|= 0.$ This naturally leads to the question: $$\textrm{ Is there a transcendental } \alpha\in(1,\infty)\textrm{ satisfying }\lim_{n\to\infty}\|\alpha^{n}\|= 0\, \textrm{?}$$ This question is highly non trivial and currently out of our reach.

In this paper, instead of studying the distribution of the sequence $\{\alpha\}, \{\alpha^{2}\},\{\alpha^{3}\},\ldots,$ we consider the distribution of the sequence $\{\alpha\},\{\alpha^{4}\}, \{\alpha^{9}\},\ldots,$ or more generally $\{\alpha^{q_{1}}\},\{\alpha^{q_{2}}\}, \{\alpha^{q_{3}}\},\ldots$ where $(q_{n})_{n=1}^{\infty}$ is a strictly increasing sequence of positive real numbers tending to infinity. We remark that if $\liminf_{n\to\infty} (q_{n+1}-q_{n})>0,$ then for almost every $\alpha\in(1,\infty)$ the sequence $(\{\alpha^{q_{n}}\})_{n=1}^{\infty}$ is uniformly distributed modulo $1$. The proof of this statement is a simple adaptation of the proof of Theorem $1.10$ in \cite{Bug}. All of the sequences we will be considering will satisfy $\liminf_{n\to\infty}( q_{n+1}-q_{n})>0.$

We are interested in the set of $\alpha\in(1,\infty)$ for which $(\{\alpha^{q_{n}}\})_{n=1}^{\infty}$ is not uniformly distributed modulo $1.$ In particular, the set of $\alpha\in(1,\infty)$ for which $\lim_{n\to\infty}\|\alpha^{q_{n}}\|= 0$, or more generally the set of $\alpha\in(1,\infty)$ which satisfy $\lim_{n\to\infty}\|\alpha^{q_{n}}-r_{n}\|= 0,$ where $(r_{n})_{n=1}^{\infty}$ is an arbitrary sequence of real numbers. Let $$E(q_{n},r_{n}):=\Big\{\alpha\in(1,\infty): \lim_{n\to\infty}\|\alpha^{q_{n}}-r_{n}\|= 0\Big\}.$$ In \cite{Dub}, Dubickas showed that whenever $\lim_{n\to\infty}(q_{n+1}-q_{n})=\infty,$ then it is possible to construct a transcendental  $\alpha$ contained in $E(q_{n},r_{n}).$ Note that the $\alpha$ they construct is always larger than $2$. Our main result is the following.

\begin{thm}
\label{Main theorem}
Let $(q_{n})_{n=1}^{\infty}$ be a strictly increasing sequence of positive real numbers satisfying $\lim_{n\to\infty}(q_{n+1}-q_{n})= \infty,$ and let $(r_{n})_{n=1}^{\infty}$ be an arbitrary sequence of real numbers. Then $E(q_{n},r_{n})$ is dense in $(1,\infty)$ and has Hausdorff dimension $1$.
\end{thm} Our proof of Theorem \ref{Main theorem} is based upon the approach of Bugeaud and Moshchevitin \cite{BugMos}, which in turn is based upon the approach of Vijayaraghavan \cite{Vij}. They show that for any $\epsilon>0$ and $(r_{n})_{n=1}^{\infty}$ a sequence of real numbers, there exists a set of Hausdorff dimension $1$ for which $\|\alpha^{n}-r_{n}\|<\epsilon$ for all $n\geq 1$. The set of $\alpha\in(1,\infty)$ which satisfy $\|\alpha^{n}-r_{n}\|<\epsilon$ for all $n$ sufficiently large is studied further in \cite{Kah}.

Given $(\{\alpha^{q_{n}}\})_{n=1}^{\infty}$ is uniformly distributed modulo $1$ for almost every $\alpha\in(1,\infty)$, Theorem \ref{Main theorem} is somewhat surprising in that it states that there exists a set, which in some sense is as large as we could hope for, which exhibits completely the opposite behaviour of uniform distribution. Indeed, taking $(r_{n})_{n=1}^{\infty}$ to be the constant sequence $r_{n}=\kappa$ for some $\kappa\in(0,1)$, we have a dense set of Hausdorff dimension $1$ satisfying $\lim_{n\to\infty}\{\alpha^{q_{n}}\}= \kappa$.
\section{Proof of Theorem \ref{Main theorem}}
We prove Theorem \ref{Main theorem} via a Cantor set construction. To help our exposition we briefly recall some of the theory from \cite{Fal} on this type of construction. Let $E_{1}\subset \mathbb{R}$ be an arbitrary closed interval and $E_{1}\supset E_{2}\supset E_{3}\supset \cdots$ be a decreasing sequence of sets, where each $E_{n}$ is a finite union of disjoint closed intervals, where each element of $E_{n}$ contains at least two elements of $E_{n+1},$ and the maximum length of the intervals in $E_{n}$ tends to $0$ as $n\to\infty$. Then the set
\begin{equation}
\label{Cantor set}
E=\bigcap_{n=1}^{\infty}E_{n}
\end{equation} is the Cantor set associated to the sequence $(E_{n})_{n=1}^{\infty}.$ The following proposition appears at Example $4.6$ in \cite{Fal}.
\begin{prop}
\label{Falconer prop}
Suppose in the construction of $E$ above each interval in $E_{n-1}$ contains at least $m_{n}$ intervals of $E_{n}$ which are separated by gaps of at least $\gamma_{n},$ where $0<\gamma_{n+1}<\gamma_{n}$ for each $n$. Then $$\dim_{H}(E)\geq \liminf_{n\to\infty} \frac{\log m_{1}\cdots m_{n-1}}{-\log m_{n}\gamma_{n}}.$$
\end{prop}

We are now is a position to prove Theorem \ref{Main theorem}.

\begin{proof}[Proof of Theorem \ref{Main theorem}]
We begin by fixing $\lambda\in(1,\infty),$ $\delta>0$ some small positive constant, and let $(r_{n})_{n=1}^{\infty}$ be our sequence of real numbers. Without loss of generality we may assume that $r_{n}\in[-1/2,1/2)$ for all $n\in\mathbb{N}$. To prove our result it is sufficient to prove that $[\lambda,\lambda+\delta]\cap E(q_{n},r_{n})$ is of Hausdorff dimension $1$.
\\

\noindent $\textbf{(1)}$ \textbf{Replacing $(q_{n})_{n=1}^{\infty}$ with $(\tilde{q}_{n})_{n=1}^{\infty}.$ }
\\ \\
Let $\epsilon>0$ be some small positive constant. We now replace our sequence $(q_{n})_{n=1}^{\infty}$ with $(\tilde{q}_{n})_{n=1}^{\infty},$ and our sequence $(r_{n})_{n=1}^{\infty}$ with $(\tilde{r}_{n})_{n=1}^{\infty}.$ We will pick our $(\tilde{q}_{n})_{n=1}^{\infty}$ and $(\tilde{r}_{n})_{n=1}^{\infty}$ in such a way that $E(\tilde{q}_{n},\tilde{r}_{n})\subset E(q_{n},r_{n}).$ We then use Proposition \ref{Falconer prop} to determine a lower bound for $\dim_{H}(E(\tilde{q}_{n},\tilde{r}_{n})\cap [\lambda,\lambda+\delta]),$ which in turn provides a lower bound for $\dim_{H}(E(q_{n},r_{n})\cap [\lambda,\lambda+\delta])$. The feature of the sequence $(\tilde{q}_{n})_{n=1}^{\infty}$ that we will exploit in our proof, is that this new sequence does not grow too fast, yet importantly we still have $\lim_{n\to\infty}(\tilde{q}_{n+1}-\tilde{q}_{n})=\infty$. The sequence $(\tilde{q}_{n})_{n=1}^{\infty}$ and the rate at which we control the growth of $(\tilde{q}_{n})_{n=1}^{\infty}$ shall depend on $\epsilon.$ For ease of exposition we drop the dependence of $(\tilde{q}_{n})_{n=1}^{\infty}$ on $\epsilon$ from our notation.

We begin our construction by asking whether
\begin{equation}
\label{epsilon growth}
q_{n+1}\leq (1+\epsilon)q_{n}
\end{equation}is satisfied for all $n\in\mathbb{N}.$ If it is, we set $(q_{n})_{n=1}^{\infty}=(\tilde{q}_{n})_{n=1}^{\infty},$ $(r_{n})_{n=1}^{\infty}=(\tilde{r}_{n})_{n=1}^{\infty}$ and stop. Suppose our sequence $(q_{n})_{n=1}^{\infty}$ doesn't satisfy (\ref{epsilon growth}) for all $n\in\mathbb{N}.$ Let $N\in\mathbb{N}$ be the first $n\in\mathbb{N}$ for which (\ref{epsilon growth}) fails. We now introduce additional terms in our sequence $(q_{n})_{n=1}^{\infty},$ situated between $q_{N}$ and $q_{N+1}$ at $$\tilde{q}^{j}_{N}:=q_{N}+j\Big(\frac{\epsilon q_{N}}{2}\Big).$$ For $j=1,\ldots,m,$ where $m$ is the smallest natural number for which $q_{N}+m(\frac{\epsilon q_{N}}{2})\in[q_{N+1}-\epsilon q_{N},q_{N+1}].$ To each $\tilde{q}^{j}_{N}$ we associate an arbitrary real number $\tilde{r}^{j}_{N}\in[-1/2,1/2)$, these terms are then placed within the sequence $(r_{n})_{n=1}^{\infty}$ between $r_{N}$ and $r_{N+1}.$ Importantly the elements $q_{N}$ and $q_{N+1}$ are still placed in the positions corresponding to $r_{N}$ and $r_{N+1}.$

The following inequalities are straightforward consequences of our construction
\begin{align}
\label{growth bound}
\tilde{q}^{1}_{N}&\leq (1+\epsilon)q_{n}\nonumber\\
\tilde{q}^{j+1}_{N}&\leq (1+\epsilon)\tilde{q}^{j}_{N} \textrm{ for } j=1,\ldots,m-1\\
q_{N+1}&\leq  (1+\epsilon)\tilde{q}^{m}_{N}.\nonumber
\end{align}In other words, all of the new terms in our sequences satisfy (\ref{epsilon growth}). The new terms in our sequence also satisfy
\begin{align}
\label{sufficient growth}
\tilde{q}^{1}_{N}-q_{N}&= \frac{\epsilon q_{N}}{2}\nonumber\\
\tilde{q}^{j+1}_{N}-\tilde{q}^{j}_{N}&=\frac{\epsilon q_{N}}{2}\textrm{   for } j=1,\ldots,m-1\\
q_{N+1}-\tilde{q}^{m}_{N}&\geq \frac{\epsilon q_{N}}{2}. \nonumber
\end{align}So if $N$ was large the gaps between successive terms in our sequences would be large. This property is what allows us to ensure $\lim_{n\to\infty}(\tilde{q}_{n+1}-\tilde{q}_{n})=\infty$.

We now take our new sequence and ask if it satisfies (\ref{epsilon growth}) for all $n\in\mathbb{N}.$ If it does then our construction is complete, and we set $(\tilde{q}_{n})_{n=1}^{\infty}$ and $(\tilde{r}_{n})_{n=1}^{\infty}$ to be our new sequences. If not, we find the smallest $n$ for which it fails and repeat the above steps. Repeating this process indefinitely if necessary, we construct sequences $(\tilde{q}_{n})_{n=1}^{\infty}$ and $(\tilde{r}_{n})_{n=1}^{\infty}$ for which
 \begin{equation}
\label{epsilon growth2}
\tilde{q}_{n+1}\leq (1+\epsilon)\tilde{q}_{n},
\end{equation} holds for all $n\in\mathbb{N}$, we retain the property
\begin{equation}
\label{sufficient growth2}
\lim_{n\to\infty}(\tilde{q}_{n+1}-\tilde{q}_{n})=\infty,
\end{equation}and
\begin{equation}
\label{inclusion}
E(\tilde{q}_{n},\tilde{r}_{n})\subset E(q_{n},r_{n}).
\end{equation}
The fact that (\ref{sufficient growth2}) holds is a consequence of (\ref{sufficient growth}). The final property (\ref{inclusion}) holds because the original terms in our sequence $(q_{n})_{n=1}^{\infty}$ keep their corresponding $r_{n}$, and at each step in our construction we only ever introduced finitely many terms between a $q_{n}$ and a $q_{n+1}$. So if $\alpha$ satisfies $\|\alpha^{\tilde{q}_{n}}-\tilde{r}_{n}\|\to 0$ then it also satisfies $\|\alpha^{q_{n}}-r_{n}\|\to 0$, i.e., (\ref{inclusion}) holds.
\\

\noindent $\textbf{(2)}$ \textbf{Construction of our Cantor set.}
\\ \\
We now construct our Cantor set $E$. Our set $E$ will be contained in $[\lambda,\lambda+\delta]\cap E(\tilde{q}_{n},\tilde{r}_{n})$ and we will be able to use Proposition \ref{Falconer prop} to obtain estimates on $\dim_{H}(E)$. We let $$\epsilon_{n}:=\frac{1}{2(\tilde{q}_{n+1}-\tilde{q}_{n})},$$ by (\ref{sufficient growth2}) we have $\epsilon_{n}\to 0$. Let us fix $\eta\in(0,1)$ some parameter that we will eventually let tend to $1$. Let $N$ be sufficiently large that
\begin{equation}
\label{2.8}
2\epsilon_{n}\lambda^{\tilde{q}_{n+1}-\tilde{q}_{n}}=\frac{\lambda^{\tilde{q}_{n+1}-\tilde{q}_{n}}}{\tilde{q}_{n+1}-\tilde{q}_{n}}\geq \lceil \lambda^{\eta(\tilde{q}_{n+1}-\tilde{q}_{n})}\rceil +2
\end{equation}for all $n\geq N$. We may also assume that this $N$ is sufficiently large that $(\lambda+\delta)^{\tilde{q}_{N}}-\lambda^{\tilde{q}_{N}}\geq 4$ and $\epsilon_{n}<1/2$ for all $n\geq N.$

We let
$$a_{n}:=\tilde{r}_{n}-\epsilon_{n} \textrm{ and } b_{n}:=\tilde{r}_{n}+\epsilon_{n}$$ for all $n\in\mathbb{N}.$ By our assumptions on $\tilde{r}_{n}$ and $\epsilon_{n},$ we may assume that $a_{n},b_{n}\in(-1,1)$ for all $n\geq N.$

Since $(\lambda+\delta)^{\tilde{q}_{N}}-\lambda^{\tilde{q}_{N}}\geq 4,$ there exists an integer $j_{N}$ for which $j_{N},j_{N}+1,\ldots, j_{N}+m+1$ are contained in $[\lambda^{\tilde{q}_{N}},(\lambda+\delta)^{\tilde{q}_{N}}],$ where $m$ is some natural number greater than or equal to $2$. We ignore the first and the last of these integers and focus on $j_{N}+1,\ldots j_{N}+m$. To each of these integers $h=j_{N}+1,\ldots j_{N}+m$ we associate the interval
$$I_{N,h}:=[(h+a_{N})^{1/\tilde{q}_{N}},(h+b_{N})^{1/\tilde{q}_{N}}].$$ By our construction each $I_{N,h}$ is contained in $[\lambda,\lambda+\delta],$ and each $\alpha\in I_{N,h}$ satisfies $\|\alpha^{\tilde{q}_{N}}-\tilde{r}_{N}\|<\epsilon_{N}.$ Let $E_{N}$ be the set of all intervals $I_{N,h}.$ For each $h$ we have
\begin{align}
\label{algorithm}
(h+b_{N})^{\tilde{q}_{N+1}/\tilde{q}_{N}}-(h+a_{N})^{\tilde{q}_{N+1}/\tilde{q}_{N}}&\geq \Big((h+b_{N})-(h+a_{N})\Big)(h+a_{N})^{\frac{\tilde{q}_{N+1}-\tilde{q}_{N}}{\tilde{q}_{N}}}\nonumber\\
&\geq 2\epsilon_{N}\lambda^{\tilde{q}_{N+1}-\tilde{q}_{N}}\\
&\geq \lceil \lambda^{\eta(\tilde{q}_{N+1}-\tilde{q}_{N})}\rceil+2.\nonumber
\end{align}Where the last inequality is by (\ref{2.8}). Therefore there exists an integer $j_{N+1}$ such that $j_{N+1},j_{N+1}+1,\ldots, j_{N+1}+\lceil\lambda^{\eta(\tilde{q}_{n+1}-\tilde{q}_{n})}\rceil+1$ are all contained in $[(h+a_{N})^{\tilde{q}_{n+1}/\tilde{q}_{n}},(h+b_{N})^{\tilde{q}_{n+1}/\tilde{q}_{n}}].$ To each $h=j_{N+1}+1,\ldots, j_{N+1}+\lceil\lambda^{\eta(\tilde{q}_{n+1}-\tilde{q}_{n})}\rceil$ we associate the interval $$I_{N+1,h}=[(h+a_{N+1})^{1/\tilde{q}_{N+1}},(h+b_{N+1})^{1/\tilde{q}_{N+1}}].$$ Importantly each interval $I_{N+1,h}$ is contained in an element of $E_{N},$ and this element contains precisely $m_{N}:=\lceil\lambda^{\eta(\tilde{q}_{N+1}-\tilde{q}_{N})}\rceil$ of these intervals. We let $E_{N+1}$ denote the set of $I_{N+1,h}.$ Any $\alpha\in I_{N+1,H}$ is contained in $[\lambda,\lambda+\delta]$ and satisfies $\|\alpha^{\tilde{q}_{N}}-\tilde{r}_{N}\|<\epsilon_{N}$ and $\|\alpha^{\tilde{q}_{N+1}}-\tilde{r}_{N+1}\|<\epsilon_{N+1}.$ We may show that (\ref{algorithm}) holds with $a_{N},b_{N}, \tilde{q}_{N},\tilde{q}_{N+1}$ replaced by $a_{N+1},b_{N+1}, \tilde{q}_{N+1},\tilde{q}_{N+2},$ and we may therefore repeat the above steps accordingly. Moreover, we may repeat the procedure described above for every subsequent $n.$ To each $n\geq N$ we let $E_{n}$ denote the set of interval $I_{n,h}$ produced in our construction. The following properties follow from our construction:
\begin{itemize}
\item Let $I_{n,h}\in E_{n},$ then for each $\alpha\in I_{n,h}$ we have $\|\alpha^{\tilde{q}_{i}}-\tilde{r}_{i}\|<\epsilon_{n}$ for $i=N,N+1,\ldots,n.$
\item $E_{n}\subset E_{n-1}\subset \cdots \subset E_{N}.$
\item $E_{n}\subset [\lambda,\lambda+\delta].$
\end{itemize}
If we let $$E=\bigcap_{n=N}^{\infty}E_{n},$$ it is clear that any $x\in E$ is contained in $ [\lambda,\lambda+\delta],$ and satisfies $\|x^{\tilde{q}_{n}}-\tilde{r}_{n}\|<\epsilon_{n}$ for all $n\geq N$, so $E\subset [\lambda,\lambda+\delta]\cap E(\tilde{q}_{n},\tilde{r}_{n}).$

It is a consequence of our construction that each element of $E_{n}$ contains exactly
\begin{equation}
\label{number of intervals}
m_{n}:=\lceil\lambda^{\eta(\tilde{q}_{n+1}-\tilde{q}_{n})}\rceil
\end{equation} elements of $E_{n+1}.$
It may also be shown that the distance between any two intervals in $E_{n}$ is always at least
\begin{equation}
\label{Gap bound}
\gamma_{n}:=\frac{c}{\tilde{q}_{n}(\lambda+\delta)^{\tilde{q}_{n}-1}},
\end{equation} where $c$ is some positive constant that is independent of $n$.

Applying Proposition \ref{Falconer prop}, combined with (\ref{number of intervals}) and (\ref{Gap bound}), we obtain the following bounds on the Hausdorff dimension of $E$:

\begin{align*}
\dim_{H}(E)&\geq \liminf_{n\to \infty} \frac{\log m\cdot m_{N}\cdots m_{n-1}}{-\log m_{n}\gamma_{n}}\\
& \geq \liminf_{n\to \infty} \frac{\log 2\cdot \lceil\lambda^{\eta(\tilde{q}_{N+1}-\tilde{q}_{N})}\rceil \cdots \lceil\lambda^{\eta(\tilde{q}_{n}-\tilde{q}_{n-1})}\rceil}{-\log \frac{c\lceil\lambda^{\eta(\tilde{q}_{n+1}-\tilde{q}_{n})}\rceil}{\tilde{q}_{n}(\lambda+\delta)^{\tilde{q}_{n}-1}}}\\
&\geq \liminf_{n\to\infty} \frac{\eta (\tilde{q}_{n}-\tilde{q}_{N})\log \lambda +\log 2}{-\log \frac{c\lceil\lambda^{\eta(\tilde{q}_{n+1}-\tilde{q}_{n})}\rceil}{\tilde{q}_{n}(\lambda+\delta)^{\tilde{q}_{n}-1}}}\\
&\geq \liminf_{n\to\infty} \frac{\eta (\tilde{q}_{n}-\tilde{q}_{N})\log \lambda +\log 2}{-\log \lceil\lambda^{\eta(\tilde{q}_{n+1}-\tilde{q}_{n})}\rceil -\log c +(\tilde{q}_{n}-1)\log (\lambda+\delta) +\log \tilde{q}_{n}}\\
&\geq \liminf_{n\to\infty} \frac{\eta (\tilde{q}_{n}-\tilde{q}_{N})\log \lambda +\log 2}{-\eta(\tilde{q}_{n+1}-\tilde{q}_{n})\log \lambda -\log c +(\tilde{q}_{n}-1)\log (\lambda+\delta) +\log \tilde{q}_{n}}\\
&\geq \frac{\eta\log \lambda}{\eta\epsilon \log \lambda + \log (\lambda+\delta)}.
\end{align*}
Since $\eta$ was arbitrary we may let it converge to $1$ so $$\dim_{H}([\lambda,\lambda+\delta]\cap E(\tilde{q}_{n},\tilde{r}_{n}))\geq  \frac{\log \lambda}{\epsilon \log \lambda +\log (\lambda +\delta)}.$$ Therefore, by (\ref{inclusion}) $$\dim_{H}([\lambda,\lambda+\delta]\cap E(q_{n},r_{n}))\geq  \frac{\log \lambda}{\epsilon \log \lambda +\log (\lambda +\delta)},$$ but since $\epsilon$ is arbitrary we may conclude that $$\dim_{H}([\lambda,\lambda+\delta]\cap E(q_{n},r_{n}))\geq \frac{\log \lambda}{\log( \lambda +\delta)}.$$ The argument we have presented also works for any $\delta'\in(0,\delta)$ and so $\dim_{H}([\lambda,\lambda+\delta']\cap E(q_{n},r_{n}))\geq \log \lambda/\log (\lambda+\delta').$  Moreover $[\lambda,\lambda+\delta']\cap E(q_{n},r_{n})\subset [\lambda,\lambda+\delta]\cap E(q_{n},r_{n}),$ so
 $$\dim_{H}([\lambda,\lambda+\delta]\cap E(q_{n},r_{n}))\geq \dim_{H}([\lambda,\lambda+\delta']\cap E(q_{n},r_{n}))\geq \frac{\log \lambda}{\log (\lambda+\delta')}.$$ Letting $\delta'$ tend to zero we deduce that $\dim_{H}([\lambda,\lambda+\delta]\cap E(q_{n},r_{n}))=1.$

\end{proof}
We conclude with a few remarks on our proof and the speed at which $\|\alpha^{q_{n}}-r_{n}\|$ converges to zero. In our proof of Theorem \ref{Main theorem} we set $\epsilon_{n}=1/2(\tilde{q}_{n+1}-\tilde{q}_{n}).$ This choice of $\epsilon_{n}$ is somewhat arbitrary, our proof still works with any sequence $\epsilon_{n}$ which tends to zero, as long as for any $\eta\in(0,1)$ we have $$2\epsilon_{n}\lambda^{\tilde{q}_{n+1}-\tilde{q}_{n}}\geq \lceil \lambda^{\eta(\tilde{q}_{n+1}-\tilde{q}_{n})}\rceil +2$$ for all $n$ sufficiently large.

 If $(q_{n})_{n=1}^{\infty}$ satisfies $\lim_{n\to\infty}q_{n+1}/q_{n}= 1$ then it is not necessary to introduce the sequences $(\tilde{q}_{n})_{n=1}^{\infty}$ and $(\tilde{r}_{n})_{n=1}^{\infty}$ in the proof of Theorem \ref{Main theorem}. This means we can say something about the speed of convergence. If $\epsilon_{n}$ decays to zero sufficiently slowly that for any $\eta\in(0,1)$ and $\lambda\in(1,\infty)$ we have $$2\epsilon_{n}\lambda^{q_{n+1}-q_{n}}\geq \lceil \lambda^{\eta(q_{n+1}-q_{n})}\rceil +2,$$ for all $n$ sufficiently large. Then the argument given in the proof of Theorem \ref{Main theorem} yields a dense set of Hausdorff dimension $1$ satisfying $$\|\alpha^{q_{n}}-r_{n}\|=O(\epsilon_{n}).$$ As an example, for any $k\in\mathbb{N}$ there exists a dense of Hausdorff dimension $1$ satisfying $$\|\alpha^{n^{2}}\|=O(n^{-k}).$$

\noindent \textbf{Acknowledgements} The author is grateful to Yann Bugeaud for pointing out \cite{Dub} and \cite{Kah}, and for some initial feedback.


\begin{thebibliography}{100}
\bibitem{Bug} Y. Bugeaud, \textit{Distribution modulo one and Diophantine approximation,}  Cambridge Tracts in Mathematics, 193. Cambridge University Press, Cambridge, 2012. xvi+300 pp. ISBN: 978-0-521-11169-0.
\bibitem{BugMos} Y. Bugeaud, N. Moshchevitin, \textit{On fractional parts of powers of real numbers close to $1$}, On fractional parts of powers of real numbers close to 1. Math. Z. 271 (2012), no. 3-4, 627–-637.
\bibitem{Dub} A. Dubickas, \textit{On the powers of some transcendental numbers,}  Bull. Austral. Math. Soc. 76 (2007), no. 3, 433–-440.
\bibitem{Fal} K. Falconer, \textit{Mathematical foundations and applications,} John Wiley \& Sons, Ltd., Chichester, 2014. xxx+368 pp. ISBN: 978-1-119-94239-9 28-01.
\bibitem{Hardy} G. H. Hardy, \textit{A problem of Diophantine approximation,} J. Indian Math. Soc. 11 (1919), 162--166.
\bibitem{Kah} J.-P. Kahane \textit{Sur la r\'{e}partition des puissances modulo 1,} C. R. Math. Acad. Sci. Paris 352 (2014), no. 5, 383–-385.  
\bibitem{Kok} J. F. Koksma, \textit{Ein mengentheoretischer Satz über die Gleichverteilung modulo Eins,} Compositio Math. 2 (1935), 250-–258.
\bibitem{Pisot} C. Pisot, \textit{Sur la r\'{e}partition modulo $1$ des puissances successives d'un m\^{e}me nombre,} C.R. Acad. Sci. Paris 204 (1937), 312--314.
\bibitem{Pisot2} C. Pisot, \textit{La r\'{e}partition modulo 1 et les nombres alg\'{e}briques,} Ann. Scuola Norm. Sup. Pisa Cl. Sci. (2) 7 (1938), no. 3-4, 205–-248.
\bibitem{Vij} T. Vijayaraghavan, \textit{On the fractional parts of powers of a number. IV,} J. Indian Math. Soc. (N.S.) 12, (1948). 33–-39.
\end{thebibliography}
\end{document}